\documentclass{amsart}

\usepackage{amsmath, amssymb, graphicx, epsfig, verbatim}
\usepackage{epic}
\usepackage{xcolor}

\usepackage{tikz}
\usetikzlibrary{arrows.meta}

\newtheorem{thm}{Theorem}[section]

\newtheorem{cor}[thm]{Corollary}
\newtheorem{lem}[thm]{Lemma}
\newtheorem{prop}[thm]{Proposition}

\newtheorem{defn}[thm]{Definition}

\numberwithin{equation}{section}

\newcommand{\Z}{\mathbb Z}

\newcommand{\R}{\mathbb R}

\newcommand{\cpkk}{{\overline {{\mathbb C}{\mathbb P}}^2}}
\newcommand{\cpk}{{\mathbb {CP}}^2}

\DeclareMathOperator{\PD}{PD}

\begin{document}

\title{Irreducible exotic $\cpk \# 7\cpkk$'s with free involutions}

\author{R. \.{I}nan\c{c} Baykur}
\address{Department of Mathematics and Statistics\\
University of Massachusetts\\
Amherst, MA 01003, USA}
\email{inanc.baykur@umass.edu}

\author{Andr\'{a}s I. Stipsicz}
\address{R\'enyi Institute of Mathematics\\
H-1053 Budapest\\ 
Re\'altanoda utca 13--15, Hungary}
\email{stipsicz.andras@renyi.hu}

\author{Zolt\'an Szab\'o}
\address{Department of Mathematics\\
Princeton University\\
 Princeton, NJ, 08544}
\email{szabo@math.princeton.edu}

\begin{abstract}
We construct infinitely many pairwise non-diffeomorphic irreducible smooth four-manifolds homeomorphic to
$\cpk \#7 \cpkk$, each admitting a free,
orientation-preserving involution. Their quotients give infinitely many
irreducible definite four-manifolds with fundamental group
$\Z/2 \Z$ and $b_2=3$, filling the corresponding gap in the
geography of such examples. The construction of
  infinitely many distinct examples relies on a novel
  equivariant torus surgery.
\end{abstract}
\maketitle

\section{Introduction}
\label{sec:intro}

The geography of smooth four-manifolds with finite fundamental group
has recently seen several new developments, especially in the definite
case; see for example
\cite{LLP,definite,definite2,finitecyclic,HNP,ArabadjiMorgan,ArabadjiMorgan2}.
A recurring strategy in the $\Z/2\Z$ case is to work equivariantly: one first
constructs exotic smooth structures on simply connected four-manifolds
equipped with compatible free, orientation-preserving involutions, and
then passes to the quotients.
For $X_n:=\cpk\#(2n+1)\cpkk$, 
quotients of free involutions on irreducible exotic smooth structures on
$X_n$ give irreducible definite four-manifolds with fundamental group
$\Z/2\Z$ and $b_2=n$.  In \cite{LLP}, the first closed exotic definite
four-manifold with non-trivial fundamental group was constructed.
Subsequently, infinite families of irreducible exotic smooth structures
on $X_n$ with free involutions were constructed in \cite{definite2} for
$n=1,2,4$.

Recall that a smooth four-manifold $X$ is \emph{irreducible} if for a
connected sum decomposition $X=X_1\# X_2$, either $X_1$ or $X_2$ is
homeomorphic to $S^4$.  For $n>4$, no irreducible smooth structure on
$X_n$ is currently known; indeed, no irreducible smooth four-manifold
with nontrivial Seiberg-Witten invariants and
$c_1^2=3\sigma+2\chi<0$ is known, and $c_1^2(X_n)=8-2n$.
Thus, in the geography of odd definite four-manifolds with
fundamental group $\Z/2\Z$ and positive second Betti number at most $4$,
the only case not covered by the constructions in \cite{definite2,ArabadjiMorgan} is
$b_2=3$.

\begin{thm}\label{thm:main}
There are infinitely many pairwise non-diffeomorphic irreducible smooth
four-manifolds $X(k)$, $k\in \Z^+$, homeomorphic to $\cpk\#7\cpkk$,
each admitting a smooth, orientation-preserving, free involution $\tau(k)$.
The family contains exactly one member that admits a symplectic structure, and
each $X(k)$ has only one pair of Seiberg-Witten basic classes.
\end{thm}

\begin{cor}\label{cor:main}
  There are infinitely many pairwise non-diffeomorphic, irreducible,
  closed, smooth four-manifolds with definite intersection form,
  fundamental group $\Z/2\Z$, and $b_2=3$.
\end{cor}

\begin{proof}
  Let $Z(k)=X(k)/\tau(k)$.
  Since $X(k)$ is simply connected and $\tau(k)$ is free, $X(k)$ is the
  universal double cover of $Z(k)$; hence $\pi_1(Z(k))\cong \Z/2\Z$.
  The Euler characteristic and signature are multiplicative under finite
  covers, so we have $\chi(Z(k))=5$ and $\sigma(Z(k))=-3$.
  Since $\pi_1(Z(k))$ is finite, $b_1(Z(k))=0$, and thus
  $b_2(Z(k))=3$.  It follows that the intersection
  form of $Z(k)$ is negative definite.

  We next prove irreducibility.  Suppose that $Z(k)=Z_1\# Z_2$.  Since
$\pi_1(Z(k))\cong \pi_1(Z_1)*\pi_1(Z_2)\cong \Z/2\Z$, one summand is
simply connected after relabeling. Pulling back the connected-sum sphere
to the universal double cover gives a connected-sum decomposition of
$X(k)$ containing two copies of this simply connected summand. Since
$X(k)$ is irreducible, this summand is homeomorphic to $S^4$. Thus the
original decomposition of $Z(k)$ was trivial.

  Finally, if $Z(k)$ and $Z(k')$ are diffeomorphic, then any such
  diffeomorphism lifts to a diffeomorphism of their universal covers.
  Hence $X(k)$ is diffeomorphic to $X(k')$, and Theorem~\ref{thm:main}
  gives $k=k'$.
\end{proof}

In Section~\ref{sec:elsoeset} we give a direct construction of
the first irreducible example, using Gompf's symplectic normal connected sum
with a compatible free involution; its irreducibility is then verified using
Seiberg-Witten theory.  In Sections~\ref{sec:infiniteexamples} and
\ref{sec:SWcomputation} we construct the infinite family, where no other
member is symplectic.  Thus the initial symplectic example is constructed
twice: first by a direct fiber-sum argument, and then again in a form suited
for equivariant torus surgery.  Starting from the genus-2 Lefschetz fibration of \cite{BKsmall}, we use the fibration picture to locate a
torus in the fiber-sum region and to control its framing and complement. This
torus is preserved setwise by the free involution, and equivariant torus
surgery along it produces the manifolds $X(k)$.

{\bf Acknowledgements.}  The first author was partially supported by
NSF grants DMS-2005327 and DMS-2506431, and by a Simons Foundation
Travel Grant.  AS was partially supported by the NKFIH Grant K146401
and by ERC Advanced Grant KnotSurf4d, and ZSz was partially supported
by the Simons Grant \emph{New structures in low dimensional topology}.
The authors thank the organizers of the 2023 \textit{Exotic
  4-manifolds} workshop at Stanford University, supported by the
Simons Collaboration Grant \emph{New structures in low dimensional
topology}, for the stimulating event at which this collaboration
began.  In the course of the prparation of this paper,
LLM was used to assist with proofreading and light language
editing.  The authors take full intellectual responsibility for the
content of this paper.


\section{Simple construction of an exotic $\cpk \# 7\cpkk$ with involution}
\label{sec:elsoeset}

The construction providing the first example of the manifolds claimed
by Theorem~\ref{thm:main} relies on a simple idea.  Consider the
four-manifold $T^2\times S^2$ (the product of the two-dimensional
torus $T^2$ and the 2-sphere $S^2$), and equip it with the projection
$pr_2\colon T^2\times S^2\to S^2$ to the second factor. For some
choice of $s\in S^2$ and $x_0\in T^2$, the submanifolds $T^2\times \{
s\}$ and $\{ x_0 \}\times S^2$ are a fiber and a section of
$pr_2$. Let $F$ denote the homology class of $T^2\times \{ s\}$.

As it is explained in detail in \cite[Section~4.2]{definite2} (the
idea originating from \cite{AP}), the homology class $2F\in
H_2(T^2\times S^2; \Z )$ of twice the fiber can be represented by a
torus, which we get by braiding along one of its $S^1$-factors.
Indeed, consider $S^1\times D \subset S^1\times S^2$, where $D\subset
S^2$ is a small disk with center $c\in S^2$. A 2-braid $B$ is an
embedding $i\colon S^1 \to  S^1\times D$ with the property that the
projection of $ S^1\times S^1\times D $ to its  $S^1$-factor in the middle,
restricted to $i(S^1)$ is a double cover.  Now the product of a
2-braid $B\subset  S^1\times D$ with $S^1$ gives the torus $S^1\times
B\subset T^2\times S^2$ with the property that its homology class is
twice the fiber.

Suppose therefore that we take such a braided surface $t$, a fiber $f$
of $pr_2$ disjoint from it and a section $\sigma$ of $pr_2$
intersecting $f$ in a single point and $t$ in two points. Smooth
$f\cap \sigma$ and one of the intersections of $t$ and $\sigma$, and
blow up the other point in $t\cap \sigma$. Blow up the resulting
genus-2 surface in two further points. In this way we get a genus-2
surface $G$ smoothly embedded in  $T^2\times S^2 \#
3\cpkk$ with trivial normal bundle.
In order to take the normal connected sum of two copies of the resulting
four-manifold along this genus-2 surface, we need to trivialize its normal
bundle, that is, fix a framing on it. In the computation of the fundamental
group of the result the choice of this framing generally plays an important
role --- the situation under consideration is, however, much simpler, and the
actual choice of the framing is not important. Therefore, fix any framing on
the surface.

With these data at hand, we now glue two copies of $T^2\times S^2 \#
3\cpkk \setminus \nu (G)$ together along their boundary through the
orientation reversing diffeomorphism $\varphi$ we describe presently.

Fix the set $s_1,t_1,s_2,t_2$ of loops on the genus-2 surface
$\Sigma _2$ generating $\pi
_1(\Sigma _2)$, as instructed by Figure~\ref{fig:maps}.  Consider the
map $\psi$ on $G=\Sigma _2$ with $\psi\circ\psi ={\rm {id}}_G$ as the
composition of two maps: The first one is the orientation reversing
involution $r$ which swaps the two torus components (when $\Sigma _2$
is viewed as the connected sum of two tori, one of them containing the
loops $s_1,t_1$, the other one the loops $s_2,t_2$).  This map has
fixed points given by the circle $\mu$ and it maps the curve $s_1$ to
$t_2$ and $t_1$ to $s_2$.  The second map is the hyperelliptic
involution $h$ (that is, 180$^\circ$ rotation around the horizontal axis),
as shown by
Figure~\ref{fig:maps}. The composition is clearly fixed point free and
orientation reversing.

\begin{figure}[htb]
\begin{center}
\setlength{\unitlength}{1mm}
 \includegraphics[height=6cm]{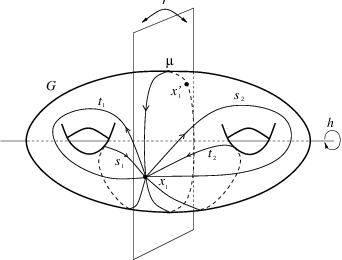}
\end{center}
\caption{\quad The map $r$ is reflection to the plane depicted (hence
  reverses the orientation of the genus-2 surface $G$), while
  $h$ is 180$^\circ$ rotation around the axis shown. Their composition is
a fixed point free, orientation reversing involution $\psi$ on $G$.}
\label{fig:maps}
\end{figure}

To define the map $\varphi$ we use the identification of the boundaries
of the two copies of $T^2\times S^2\# 3\cpkk \setminus \nu (G)$
with $G\times S^1$ using the framing (described earlier),
where the curves
$s_1,t_1,s_2,t_2$ lie on a copy of $G\times \{p\}$ (for some
point $p\in S^1$).
Then $\varphi$ is given by
\begin{equation}
  \label{eq:Psi}
\varphi(g,\theta) = (\psi(g), \theta)
\end{equation}
for $g \in G$, $\theta \in S^1$.

In determining the action of the involution $\psi$ on
$\pi _1 (G, x_1)$, we need to set up a further convention.
Suppose that the two points $x_1$ and $x_1'$ are opposite points of
the circle $\mu$ we get by intersecting $G$ with the plane to which $r$
reflects. (Note that $\mu$ can be viewed
as a loop based either at $x_1$ or at $x_1'$.) Then
the map $\psi$ induces a map $\psi _*\colon \pi _1(G, x_1)\to \pi _1(G,x_1')$.
Connecting $x_1'$ to $x_1$ by one of the semi-circles of $\mu$
(dictated by the orientation of $\mu$ given in Figure~\ref{fig:maps}), we can identify
$\pi _1(G, x_1')$ with $\pi _1(G, x_1)$ in such a way that when composing
$\psi _*$ with this identification, we get that
\begin{equation}\label{eq:action}
  s_1\mapsto t_2^{-1}\mu , \qquad t_1 \mapsto \mu ^{-1} s_2^{-1},
  \qquad s_2\mapsto \mu ^{-1}t_1^{-1},
  \qquad t_2\mapsto s_1^{-1}\mu.
  \end{equation}

In our later arguments we will need to know which elements of the
fundamental group of the boundary $G\times S^1$ map to zero in the complement
$T^2\times S^2\# 3\cpkk \setminus \nu (G)$. The exceptional spheres
of the last two blow-ups meet $G$ once, so a normal meridian of $G$ is trivial
in the complement. Consequently, the inclusion-induced map identifies
the fundamental group of the complement with the fundamental group of
$T^2\times S^2\# 3\cpkk$, an abelian group isomorphic to
$\pi _1 (T^2\times S^2)\cong \Z ^2$.
Thus we can work with homologies rather than homotopies.
The generators of the fundamental group of
$T^2\times S^2\# 3\cpkk \setminus \nu (G)$ 
can be chosen to be the two main circles
of $T^2$, which we call $a$ and $b$. Under the above identification,
the inclusion-induced map is given by
\[
s_1\mapsto a, \qquad t_1,t_2\mapsto b, \qquad  s_2\mapsto a^2.
\]
The loop $\mu$ appearing in Equation~\eqref{eq:action} maps trivially: it is represented on $G$ by the commutator $[s_1,t_1]$, and the
complement group is abelian.

Consider now the four-manifold
\begin{equation}\label{eq:defi}
X=(T^2\times S^2 \# 3\cpkk \setminus \nu (G)) \cup _{\varphi} (T^2\times S^2 \#
3\cpkk \setminus \nu (G)) .
\end{equation}


\begin{lem}\label{lem:char}
  The characteristic numbers (the Euler characteristic $\chi$ and
  the signature $\sigma$) of $X$ are given as
  \[
  \chi (X)= 10, \qquad \sigma (X)= -6.
  \]
\end{lem}
\begin{proof}
  As $\chi (T^2\times S^2)=0$, it follows that
  $\chi( T^2\times S^2 \# 3\cpkk \setminus \nu (G))=5$, hence by additivity of
  $\chi$ the claim follows. In a similar manner,
  $\sigma (T^2\times S^2)=0$ and so 
  $\sigma( T^2\times S^2 \# 3\cpkk \setminus \nu (G))=-3$, so by additivity again
  we get that $\sigma (X)=-6$.
  \end{proof}

\begin{prop}\label{prop:sc}
  The four-manifold $X$ is simply connected.
\end{prop}
\begin{proof}
  Let lower case letters denote elements of the fundamental group in
  one copy of $T^2\times S^2\# 3\cpkk-\nu (G)$, while the
  corresponding upper case letters denote elements in the other copy of the same
  manifold.  By the Seifert-Van Kampen theorem the fundamental group
  $\pi _1 (X)$ is generated by the generators $a,b$ and $A,B$
  generating the fundamental groups of the two copies of $T^2\times
  S^2\# 3\cpkk$, together with the normal meridians of the surfaces
  $G$ and $G'$ in the two copies.  The blow-ups show that these
  normal meridians are trivial in $\pi _1$,
  consequently in the following we will ignore their classes.
 
  The loop $\mu\subset G$ appearing in Equation~\eqref{eq:action}
  also maps trivially in the complements, as noted above.  Therefore the
  identification of $s_1$ with $T_2^{-1}$ by the gluing map implies
  that $a$ is equal to $B^{-1}$. The identification of $s_2$ with $T_1^{-1}$
  identifies $a^2$ with $B^{-1}$ in the same way, implying that $a=1$ and
  so $B=1$. Similarly, using the identifications of $S_1^{-1}$ with $t_2$
  and $S_2^{-1}$ with $t_1$, we get
  $A^{-2}=b=A^{-1}$, implying $A=1$ and $b=1$, concluding the argument.
    \end{proof}

\begin{thm}
  The four-manifold $X$ admits a symplectic structure and it is homeomorphic
  to $\cpk \# 7 \cpkk$. It admits a free, orientation preserving involution
  $\tau$.
\end{thm}
\begin{proof}
  The manifold $T^2\times S^2$, and its thrice blow-up both admit
  symplectic structures. The symplectic form can be chosen so that the
  fibers and the sections are symplectic submanifolds, and the toric
  representative $t$ of $2F$ can be chosen to be symplectic as
  well. As the smoothing and the blow-ups preserve this property, $G$
  is a symplectic submanifold of $T^2\times S^2\# 3\cpkk$. The gluing
  map $\varphi$ can be viewed as a symplectic normal connected sum
  map, hence by \cite{Gompf} it follows that $X$ is symplectic.
  Indeed, by fixing the symplectic structure $\omega $ on one copy,
  and $-\omega$ on the other copy of $T^2\times S^2\# 3\cpkk$, the map
  $\psi$ is an orientation preserving map between the two copies of
  $G$ (each oriented by the respective symplectic structures), and can
  be chosen to be a symplectomorphism.
  
  The signature and Euler characteristic of $X$ have been determined in
  Lemma~\ref{lem:char}, giving $\sigma (X)=-6$ and $\chi (X)=10$. As
  $X$ is simply connected by Proposition~\ref{prop:sc}, the
  application of Freedman's theorem \cite{Fr} implies that $X$ is
  homeomorphic to $\cpk \# 7\cpkk$.  The orientation preserving
  involution $\tau$ is given by patching the identification of the two
  manifolds-with-boundary in the definition of
  Equation~\eqref{eq:defi} with the gluing map $\varphi$. Since
  $\psi$, and hence $\varphi$, is fixed point free, the resulting
  involution is free.
\end{proof}

\begin{prop}
  The symplectic four-manifold $X$ is an irreducible, exotic copy of $\cpk \# 7\cpkk$.
\end{prop}
\begin{proof}
  The minimality of the symplectic four-manifold $X$ follows
  from Usher's \cite[Theorem~1.1]{usher}. Since $X$ is simply connected,
  the theorem of Hamilton--Kotschick \cite{HK} implies that $X$ is
  irreducible, verifying the claim.


    Alternatively, we can determine the (small perturbation)
  Seiberg-Witten invariants of $X$. Note that although
  $b_2^+(X)=1$, since $b_2^-(X)=7\leq 9$, the small perturbation
  Seiberg-Witten invariant, that is, the invariant corresponding to the
  metric chamber, is a well-defined smooth invariant.

  Before starting the computation, let us denote the homology
  class of $[T^2\times \{ pt.\}]$ by $f$, and of $[\{ pt. \}\times S^2]$
  by $s$ in $H_2(T^2\times S^2; \Z )$.
  The first Chern class $c_1$ of $T^2\times S^2$ is the Poincar\'e dual of
  twice the homology class $f$; equivalently, it evaluates on $f$ as $0$
  and on $s$ as $2$.
  The first Chern class of
  $T^2\times S^2\# 3\cpkk$ is then $c_1-E_1-E_2-E_3$, where $E_i$ is
  the Poincar\'e dual of the exceptional sphere $e_i$ of the $i^{th}$
  blow-up, and satisfies $E_i(e_j)=-\delta _{ij}$.

  Before proceeding any further, let us recall that a metric on a
  manifold $X$ provides the Hodge identification of $H^2(X; \R) $ with
  the space of harmonic 2-forms.  In case the 4-manifold $X$ has
  $b_2^+=1$, the self-dual harmonic 2-forms provide a line. We can
  intersect this line with those harmonic 2-forms which have square 1,
  a subset in $H^2(X; \R)$ having two components. Therefore the line
  of self-dual 2-forms intersects this subset in two points.
  
  The Seiberg-Witten invariant of $T^2\times S^2\# 3\cpkk$ with a
  metric of period point Poincar\'e dual to
  $f+s+\epsilon _1e_1+\epsilon _2 e_2+\epsilon _3e_3$, where the
  $\epsilon _i$ can be chosen arbitrarily small, is, by the blow-up
  formula, equal to the Seiberg-Witten invariant of $T^2\times S^2$,
  hence it is zero.  Thus in this chamber the Seiberg-Witten invariant
  for the spin$^c$ structure with first Chern class
  $c_1-E_1-E_2-E_3$ is still zero, and the pairing of this class with
  the period point is positive.

  When introducing a metric on $T^2\times S^2\# 3\cpkk$ which pulls
  the surface $G$ far away, the period point gets closer and closer to
  the ray Poincar\'e dual to the homology class of $G$, namely to
  $\pm (3f+s-2e_1-e_2-e_3)$. The correct sign is $+$, as this is the
  one which is in the boundary of the component containing the above
  period point. The pairing of $c_1-E_1-E_2-E_3$ with this class is
  negative. Hence, when passing to the neck-stretching chamber, we cross
  a wall, and by the wall-crossing formula the Seiberg-Witten invariant
  is $\pm 1$ on this spin$^c$ structure.
    
  Finally we appeal to the gluing of these invariants along $\Sigma
  _2\times S^1$ in the top spin$^c$ structure, that is, in the structure
  evaluating on $\Sigma _2$ as $\pm 2$, from \cite{MSzT}, where the
  relevant Floer homology of the three-manifold is 1-dimensional. This
  argument implies that $X$ has a spin$^c$ structure $K$ with non-trivial
  Seiberg-Witten invariant in the metric chamber, showing that it is not
  diffeomorphic to $\cpk \# 7\cpkk$.

  In order to compute all the basic classes of $X$ we will use the
  adjunction inequalities for smoothly embedded genus-2 surfaces with
  square 0. First note that we have homology classes $A,B$ of square 0,
  where $A$ is the homology class of $G$ and $B$ corresponds to gluing
  together two surfaces of the type $T^2\times \{ pt. \}$ from the two
  $T^2\times S^2\#3\cpkk$ sides. $A$ and $B$ are both represented by surfaces
  of genus 2. There is another homology class $C$ given by gluing
  together two copies of $\{ pt. \}\times S^2$ from the two sides. We also have six homology
  classes of square $-1$. Three of them $x_1,x_2,x_3$ come from one of
  the copies $T^2\times S^2\#3\cpkk-G$, and in
  $T^2\times S^2\#3\cpkk$ their Poincar\'e duals represent
  $f-e_1$, $f-e_2$, $2f-e_3$, respectively. The other three classes
  $x_4,x_5,x_6$ are symmetric and come from the other copy. Now
  $A,B,x_1,\ldots,x_6$ give a splitting of $H_2(X;\Z)$ into a hyperbolic
  pair and a negative definite diagonal intersection form.

  Since a basic class $L$ is characteristic, we have
  \[
  L = 2k_1 \cdot \PD(A) + 2k_2 \cdot \PD(B) +
  \sum_{i=1}^6 t_i\PD(x_i),
  \]
  where all $t_i$ are odd. Using the adjunction inequality for $L$ and
  the surfaces representing $A$ and $B$, we get $|k_1|\leq 1$ and
  $|k_2|\leq 1$. Using that the formal dimension of $L$ has to be
  non-negative, we have $|t_i|=1$ for all $i$, and either
  $k_1=k_2=1$ or $k_1=k_2=-1$. By multiplying with $\pm 1$, we can
  assume that $k_1=k_2=1$.

  To finish the argument, using intersection numbers with the elements of  the above
  basis, we get that
  \[
  C=2A+3B-x_1-x_2-2x_3-x_4-x_5-2x_6.
  \]
  Since $C$ can be represented by a smooth genus-2 surface with square
  0, it follows that $|L(C)|\leq 2$. However, plugging in the formula for
  $L$ together with the intersection form computation for $X$, we get
  \[
  L(C)=10-t_1-t_2-2t_3-t_4-t_5-2t_6.
  \]
  This together with $|t_i|=1$ implies that all $t_i=1$. Hence there are
  at most two basic classes, and therefore $\pm K$ are the only basic
  classes of $X$. This gives another verification of the irreducibility
  of $X$.
\end{proof}

\section{Infinitely many examples}
\label{sec:infiniteexamples}
The above construction admits a modification leading to infinitely
many irreducible examples, and eventually to the proof of
Theorem~\ref{thm:main}, which we discuss now.  The new ingredient is an
equivariant torus surgery construction.  We first locate a torus in the
fiber-sum region of an intermediate four-manifold with $\pi_1\cong \Z$;
the torus is preserved setwise by the free involution, and surgery in the
distinguished direction kills the remaining fundamental group.  Equivalently,
after the first surgery, the construction may be viewed as varying the
coefficient of torus surgery on a nullhomologous torus in an exotic
$\cpk\#7\cpkk$ with a compatible free involution.

To set the stage, we use the following positive factorization, which is a
Hurwitz-equivalent presentation of the smallest genus-2 Lefschetz fibration
of Baykur--Korkmaz \cite[Theorem~7]{BKsmall}; see also
\cite[Figure~4]{Huang} for the curves in this notation.  The total space
$M$ is diffeomorphic to $T^2\times S^2\#3\cpkk$ by
\cite[Proposition~9]{BKsmall}; in particular,
$\chi(M)=3$, $\sigma(M)=-3$, and $\pi_1(M)\cong \Z\oplus \Z$.

\begin{figure}[htb]
\begin{center}
\includegraphics[width=\textwidth]{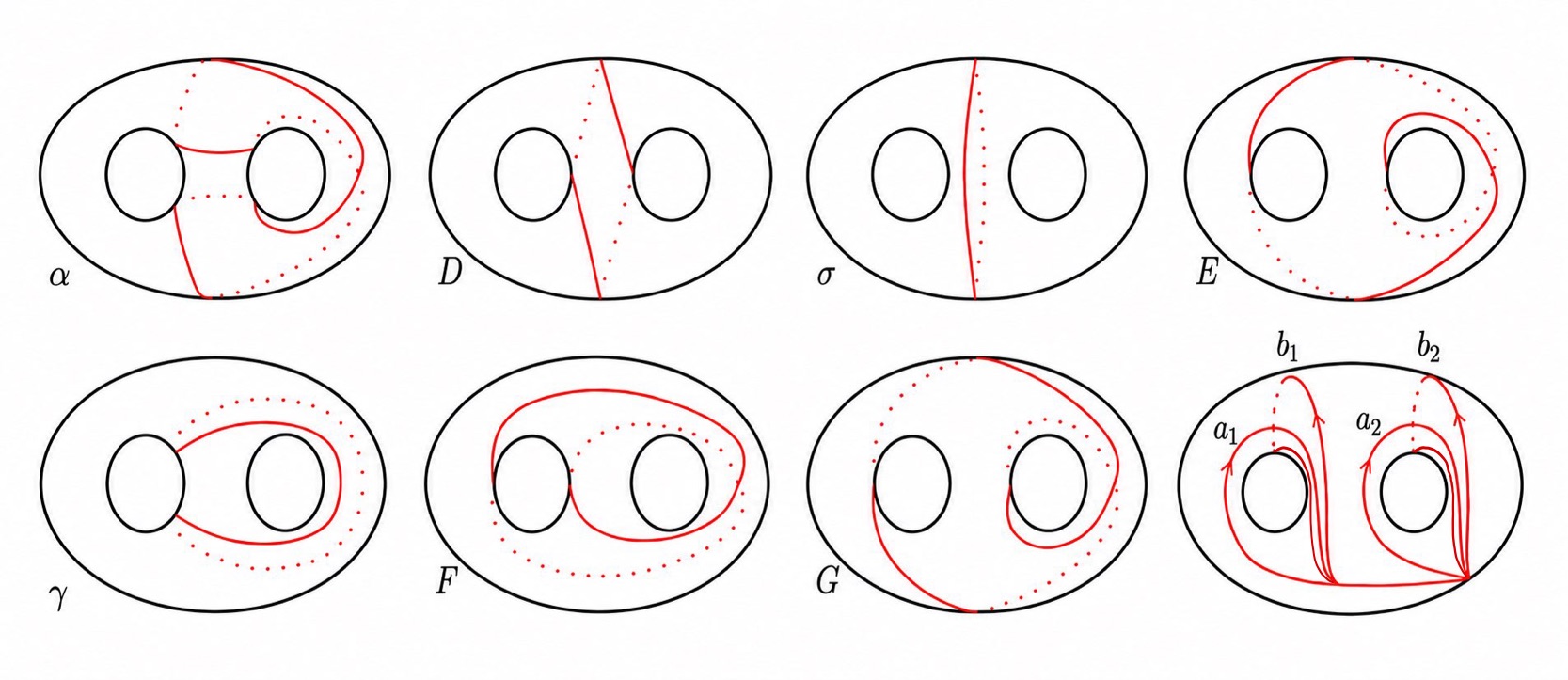}
\end{center}
\caption{The monodromy curves for the genus--$2$ Lefschetz fibration on $M=T^2\times S^2\#3\cpkk$ and the oriented, based curves $a_i, b_i$ generating $\pi_1(\Sigma_2)$.}
\label{fig:curves}
\end{figure}

For the curves $D,E,F,G,\alpha,\sigma,\gamma$ shown in
Figure~\ref{fig:curves}, we have the positive factorization
\begin{equation}\label{eq:relation}
t_{\alpha}t_Dt_{\sigma}t_Et_{\gamma}t_Ft_G=1
\end{equation}
in the mapping class group of the closed genus-2 surface.

We will use the following description of the fundamental group. Since the
fibration has a section, the normal meridian of a regular fiber is trivial
in the fiber complement. Hence the group relevant for the fiber-sum
calculation, which is the quotient
$\pi_1(\Sigma_2)/N$,
where $N$ denotes  the normal subgroup
generated by the vanishing cycles $D,E,F,G,\alpha,\sigma,\gamma$, is the same as $\pi_1(M)\cong \Z^2$. Since it is abelian, the Seifert--Van Kampen
calculations below can be carried out at the level of homology. 

Reading
the nonseparating vanishing cycles in $H_1(\Sigma_2;\Z)$, we get the two
independent relations
$$
        b_1+b_2=0,\qquad a_1+2a_2=0
$$
where $a_i, b_i$ denote the (oriented) curves generating $\pi_1(\Sigma_2)$; see Figure~\ref{fig:curves}.
Thus this abelian group is freely generated by $a_2$ and $b_1$. 

Let $\phi\colon \Sigma_2\to\Sigma_2$ be the composition of the
reflection in the horizontal plane with the hyperelliptic involution.
With respect to the generating curves $a_1,b_1,a_2,b_2$, the induced
map on homology is
$$
        a_1\mapsto a_1,\qquad b_1\mapsto -b_1,\qquad
        a_2\mapsto a_2,\qquad b_2\mapsto -b_2 .
$$
We define the gluing map
$$
        \Phi=a\times (t_{b_1}\circ\phi)
        \colon S^1\times \Sigma_2\to S^1\times \Sigma_2,
$$
where $a$ is the antipodal map on the $S^1$-factor. Here we regard $b_1$ as a free (unbased) loop so that $\phi$ maps $b_1$ back to itself. We use $\Phi$
to take the twisted fiber sum of two copies of the Lefschetz fibration
$M\to S^2$, and denote the resulting four-manifold by $X'$.

\begin{prop}\label{prop:xprime}
  The four-manifold $X'$ has $\pi _1(X')\cong \Z$.
\end{prop}
\begin{proof}
  We apply the Seifert-Van Kampen theorem.  By the discussion above,
  the fundamental groups of the fiber complements are abelian, and the
  normal meridians of the fibers are trivial.  It is therefore enough to
  work with the induced homology relations.

  Let $a_i, b_i$ denote the generators on the first side, and
  $A_i, B_i$ the corresponding generators on the second summand.
  On each side, the vanishing cycles yield the relations
  $$
        b_1+b_2=0,\qquad a_1+2a_2=0
  $$
  and
  $$
        B_1+B_2=0,\qquad A_1+2A_2=0
  $$
  in the fundamental group. Under the boundary identification, and with our convention for the right-handed
  Dehn twist $t_{b_1}$, the Picard--Lefschetz formula gives
  $$
        a_1\mapsto A_1+B_1,\qquad a_2\mapsto A_2,\qquad
        b_1\mapsto -B_1,\qquad b_2\mapsto -B_2 .
  $$
  Therefore the first-summand relation $a_1+2a_2=0$ maps to
  $$
        A_1+2A_2+B_1=0.
  $$
  Comparing this with the second-summand relation $A_1+2A_2=0$, we get
  $B_1=0$, and hence also $B_2=0$.  The gluing then gives
  $b_1=b_2=0$.  The remaining relations identify $a_1=-2a_2$ and
  $A_1=-2A_2$, while $a_2$ is identified with $A_2$.  Thus the
  fundamental group is infinite cyclic, generated by $a_2$.
\end{proof}

The Euler characteristic and signature of the four-manifold $X'=M\# _{\Phi}M$ are determined by the fiber-sum formula.
\begin{lem}
  The four-manifold $X'$ has $\chi (X')=10$,
  $\sigma (X')=-6$, and hence $b_2^+(X')=2$ and $b_2^-(X')=8$. \qed
  \end{lem}

Notice that $\Phi$ is a fixed point free involution of the gluing
three-manifold. As $\phi$ is orientation-reversing and sends $b_1$
setwise to itself, we have
$$
        \phi \, t_{b_1} \, \phi^{-1}=t_{b_1}^{-1}.
$$
It follows that $t_{b_1}\circ\phi$ is an involution.  Since the antipodal
map $a$ on the $S^1$-factor is fixed point free,
$\Phi=a\times(t_{b_1}\circ\phi)$ is fixed point free.  Consequently,
$X'$ admits a fixed point free involution $\iota _{X'}$ by flipping the two sides of the twisted fiber sum.

Consider the torus $T\subset \partial(M\setminus \nu F)\subset X'$
obtained by taking the product of $a_2\subset F\cong \Sigma_2$ with the
meridian circle of the removed fiber.

\begin{lem}\label{lem:dual-sphere}
  The torus $T$ admits an embedded dual sphere $S\subset X'$ with
  $S^2=-2$. In particular, $T$ is homologically essential, and the normal
  meridian of $T$ is trivial in $\pi_1(X'\setminus \nu T)$.
\end{lem}
\begin{proof}
  Let $D'=t_{\sigma}^{-1}(D)$. Applying a Hurwitz move to the consecutive factors $t_Dt_{\sigma}$ in the monodromy factorization for $M$ given in Equation~\eqref{eq:relation}, we get a new factorization where this subword is replaced by 
  $t_{\sigma}t_{D'}$. Hence $D'$ can be taken as a vanishing cycle in the same Lefschetz fibration.

  The curve $D$ intersects $a_2$ once. Moreover, from Figure~\ref{fig:curves}
  the fiber part $t_{b_1}\circ\phi$ of the gluing map sends $D$ to $D'$:
  the reflection $\phi$ takes $D$ to $D'$, and $D'$ is disjoint from $b_1$.
  Choose a thimble for $D$ in one summand and for $D'$ in the
  other summand.  In the fiber-sum neck, the rotation in the meridian
  direction supplies a cylinder joining the copy of $D$ on one boundary
  fiber to the copy of $D'$ on the boundary fiber to which it is identified.
  The two thimbles together with this cylinder form an embedded sphere
  $S\subset X'$.

  This is the usual sphere associated to the two matching vanishing
  cycles, so its self-intersection is $-2$. After a small perturbation in
  the neck, its intersection with $T=a_2\times \mu_F$ is transverse and
  consists of one point, since $D$ intersects $a_2$ once. Thus $S$ is dual
  to $T$. Removing a small disk from $S$ at this intersection point gives a
  disk in $X'\setminus \nu T$ bounded by the normal meridian of $T$, so this
  meridian is trivial in $\pi_1(X'\setminus \nu T)$.
\end{proof}

The next result will play a key role for our equivariant torus surgery:

\begin{prop}
  The torus $T$ is setwise fixed by the involution $\iota _{X'}$,
  and admits a framing $f$ so that the induced involution on
  $T^3=\partial(\nu T)$ is $a\times a\times a$, where
  $a\colon S^1\to S^1$ is the antipodal map.
\end{prop}
\begin{proof}
  The curve $a_2$ is setwise fixed by $t_{b_1}\circ\phi$, since
  $\phi$ fixes $a_2$ setwise and $a_2$ is disjoint from $b_1$.
  The meridian circle of the fiber is acted on by the antipodal map in
  the gluing region.  Hence $T$ is setwise fixed by $\iota _{X'}$.

  Choose the framing $f$ from the product coordinates in the fiber-sum
  neck.  In these coordinates the involution acts antipodally on the two
  circle factors of $T$.  The normal directions to $T$ are given by the
  normal direction to $a_2$ in $\Sigma_2$ and the neck direction; both are
  reversed by the side-switching involution.  Thus the induced involution
  on $\partial(\nu T)$ is $a\times a\times a$.
\end{proof}

\begin{defn}\label{def:TorusSurgery}
Define $X(k)$ as the result of torus surgery on $X'$ along the torus
$T$ with framing $f$, distinguished curve $a_2$, and coefficient
$k\in \Z^+$.  Equivalently, if $m_T$ denotes the normal meridian of $T$,
then the meridian of the attached $D^2\times T^2$ is glued to the curve
$a_2+k m_T$ on $\partial(\nu T)$.
\end{defn}

\begin{prop}\label{prop:extends}
  The free, orientation-preserving involution $\iota _{X'}$ extends from
  the complement of $\nu(T)$ to $X(k)$ as a free, orientation-preserving
  involution $\tau(k)$.
\end{prop}
\begin{proof}
  We use the framing $f$ from the previous proposition to identify
  $\partial(\nu T)$ with
  $T^3=S^1_x\times S^1_y\times S^1_z$, so that the induced involution is
  $a\times a\times a$.  Here $x$ is the distinguished curve $a_2$, $y$
  is the normal meridian of $T$, and $z$ is the remaining circle factor.
  In the torus surgery defining $X(k)$, we attach $D^2\times T^2$ so that
  the primitive curve $x+ky$ bounds the meridian disk.

  The boundary involution $a\times a\times a$ is translation by
  $(1/2,1/2,1/2)$ on $\R^3/\Z^3$, and hence induces the identity on
  $H_1(T^3;\Z)$.  In the basis
$$
        u=x+ky,\qquad v=y,\qquad z=z,
$$
this translation is given by
$$
        (u,v,z)\mapsto
        \bigl(u+(1+k)/2,\; v+1/2,\; z+1/2\bigr).
$$
Therefore it preserves the filling slope $u=x+ky$ and extends over
$D^2\times T^2$ by rotating the meridian disk through angle
$\pi(1+k)$ and translating the two remaining circle factors as above.
This extension is orientation-preserving and fixed point free, since the
$z$-coordinate is translated by $1/2$.

  Gluing this local extension to the restriction of $\iota _{X'}$ on
  $X'\setminus \nu T$ gives the required free, orientation-preserving
  involution $\tau(k)$ on $X(k)$.
\end{proof}

For $k=1$ the torus surgery can be seen to be a Luttinger surgery along the Lagrangian torus $T$, and thus, $X(1)$ is symplectic---and likely diffeomorphic to the example we constructed in Section~\ref{sec:elsoeset}.

The characteristic numbers of $X(k)$ are unchanged by the torus surgery.

\begin{prop}
  The four-manifolds $X(k)$ are simply connected, and satisfy
  $\chi(X(k))=10$ and $\sigma(X(k))=-6$ for all $k\in \Z^+$.
\end{prop}
\begin{proof}
  By Proposition~\ref{prop:xprime}, the fundamental group of $X'$ is
  generated by $a_2$.  The dual sphere to $T$ shows that the normal
  meridian $m_T$ is trivial in the complement of $T$.  With the surgery
  convention of Definition~\ref{def:TorusSurgery}, the surgery adds the
  relation $a_2+k m_T=0$, and hence kills $a_2$.  Therefore $X(k)$ is
  simply connected.  Since torus surgery does not change the Euler
  characteristic or the signature, the claim follows.  In particular,
  $b_2^+(X(k))=1$ and $b_2^-(X(k))=7$.
\end{proof}

By Freedman's classification theorem \cite{Fr}, we get:
\begin{cor}\label{cor:TheTopology}
  For each $k \in \Z^+$, $X(k)$ is homeomorphic to $\cpk \#7\cpkk$.
\end{cor}


\section{Seiberg-Witten invariant computations}
\label{sec:SWcomputation}

In this section we complete the proof of Theorem~\ref{thm:main}.  The
topology and the free involutions were constructed in the previous section;
it remains to determine the Seiberg-Witten invariants which then
distinguish the smooth structures and also verify irreducibility.

We first identify the basic classes of $X'$, the four-manifold with
$\pi_1(X')\cong \Z$ constructed before the torus surgeries.

\begin{prop}\label{prop:basicXprime}
  The symplectic four-manifold $X'$ has exactly two basic classes.
\end{prop}
\begin{proof}
  Since $b_1(X')=1$, $\chi(X')=10$, and $\sigma(X')=-6$, we have
  $b_2^+(X')=2$ and $2\chi(X')+3\sigma(X')=2$.  By Taubes' theorem
  \cite{Taubes}, $X'$ has Seiberg-Witten simple type, and
  $\pm c_1(X')$ are basic classes.

  We use the following basis for $H_2(X';\Z)$.  First, we have the torus
  $T$ and the embedded dual sphere $S$ of Lemma~\ref{lem:dual-sphere};
  their intersection matrix is
  $$
  \begin{pmatrix}
  0 & \ 1\\
  1 & -2
  \end{pmatrix}.
  $$
  Next, let $F$ denote the regular fiber.  Gluing the sections of the two
  copies of $M$ gives a sphere $Q$ with $Q^2=-2$ and $F\cdot Q=1$.
  Let $R$ be the genus-2 surface of square zero obtained by smoothing the positive double point intersection in
  $F\cup Q$.  Then $F^2=R^2=0$ and $F\cdot R=1$.  Finally, from each of
  the six reducible singular fibers of $X'$ choose the genus-one component
  $T_i$ disjoint from $Q$.  Then $T_i^2=-1$, and the other component of
  the corresponding reducible fiber represents $T_i'=F-T_i$.  With these
  choices, all other intersections vanish, and the resulting intersection
  matrix is unimodular.  Hence these ten classes form a basis.

  Let $K$ be a basic class, and write its Poincar\'e dual in this basis as
  $$
  K=rT+sS+aF+bR+\sum_{i=1}^6 t_iT_i.
  $$
  The adjunction inequality applied to the square-zero torus $T$ gives
  $K(T)=0$.  Smoothing the single intersection point of $T$ and $S$ gives
  an embedded torus representing $T+S$ and having square zero.  Applying
  adjunction to this torus gives $K(T+S)=0$.  Hence $K(S)=0$ as well,
  and therefore $r=s=0$.

  Since $K$ is characteristic, the coefficients $a,b$ are even and the
  coefficients $t_i$ are odd.  The adjunction inequality applied to the
  genus-2 square-zero surfaces $F$ and $R$ gives
  $$
  |K(F)|=|b|\leq 2,\qquad |K(R)|=|a|\leq 2.
  $$
  Applying adjunction to the genus-one surfaces $T_i$ of square $-1$ gives
  $|t_i|\leq 1$.  Thus $t_i=\pm 1$ for all $i$.

  Since $X'$ has simple type, every basic class satisfies
  $K^2=2\chi(X')+3\sigma(X')=2$.  Therefore
  $$
  2=K^2=2ab-\sum_{i=1}^6 t_i^2=2ab-6,
  $$
  and hence $ab=4$.  Since $a$ and $b$ are even and $|a|,|b|\leq 2$,
  we have either $a=b=2$ or $a=b=-2$.  Replacing $K$ by $-K$ if necessary,
  assume that $a=b=2$.

  It remains to determine the signs of the $t_i$.  For each $i$, the other
  component $T_i'=F-T_i$ of the corresponding reducible fiber is also a
  genus-one surface of square $-1$.  Adjunction gives
  $
  |K(T_i')|\leq 1.
  $, but
  $$
  K(T_i')=K(F-T_i)=2+t_i,
  $$
  and since $t_i=\pm 1$, this forces $t_i=-1$ for every $i$.  Thus, up to
  sign, there is only one possible choice for the signs:
  $$
  K=2F+2R-\sum_{i=1}^6 T_i.
  $$
  Since $\pm c_1(X')$ are basic classes, these are precisely the two basic
  classes of $X'$.
\end{proof}

\begin{proof}[Proof of Theorem~\ref{thm:main}]
  By Corollary~\ref{cor:TheTopology}, the manifolds $X(k)$ are
  homeomorphic to $\cpk\#7\cpkk$, and by
  Proposition~\ref{prop:extends} they admit free, orientation-preserving
  involutions.  Applying the Morgan--Mrowka--Szab\'o torus-surgery formula
  \cite{MMSz} to the surgery family determined by $T$, and using
  Proposition~\ref{prop:basicXprime}, we get that $X(k)$ has exactly two
  Seiberg-Witten basic classes $\pm K_k$, and
  $$
        SW_{X(k)}(\pm K_k)=\pm k.
  $$
  Since the small perturbation Seiberg-Witten invariant is a
  diffeomorphism invariant for manifolds homeomorphic to
  $\cpk\#7\cpkk$, the manifolds $X(k)$ and $X(k')$ are diffeomorphic only
  if $k=k'$.

  The nontrivial Seiberg-Witten invariants show that these smooth
  structures are exotic.  The uniqueness of the pair of basic classes also
  gives irreducibility: a nontrivial connected-sum decomposition would,
  by Donaldson's diagonalization theorem \cite{Donaldson} and the standard
  connected-sum/blow-up formula for Seiberg-Witten invariants, produce
  more than one pair of basic classes.  Finally, the symplectic member is
  $X(1)$; for $k>1$, Taubes' theorem would force a basic class with
  Seiberg-Witten invariant equal to $\pm 1$, while the only nonzero values
  above are $\pm k$.  Hence no $X(k)$ with $k>1$ admits a symplectic
  structure.
\end{proof}

\section{Final remarks}
\label{sec:final}

We end with two comments.  The first concerns related constructions.  Similar
equivariant normal-sum constructions can be carried out starting from
$T^4=T^2\times T^2$.  Take a fiber and a section, smooth their single
intersection point, and blow up twice.  This gives a genus-2 surface
$\Sigma\subset T^4\#2\cpkk$ with trivial normal bundle.  Using an
appropriate framing and the same boundary gluing as in
Section~\ref{sec:elsoeset}, one obtains a four-manifold
$$
W=(T^4\#2\cpkk \setminus \nu(\Sigma))\cup
(T^4\#2\cpkk \setminus \nu(\Sigma))
$$
with an orientation-preserving free involution.

In this case the resulting manifold is not simply connected.  Applying
appropriate torus surgeries on tori in $T^4$ disjoint from the section
and the fiber used in the construction of $\Sigma$, and doing so
equivariantly, gives simply connected four-manifolds homeomorphic to
$\cpk\#5\cpkk$.  With suitable surgery coefficients, the resulting
infinite family can be distinguished by Seiberg-Witten invariants.  Their
quotients give infinitely many exotic irreducible definite four-manifolds
with $\pi_1=\Z/2\Z$ and $b_2=2$.

Starting with a braided torus and a section, the same construction can be
carried out using only one blow-up: the genus-2 surface $\Sigma$ with
trivial normal bundle lies in $T^4\#\cpkk$, and the final construction
provides an infinite family of exotic irreducible
$\cpk\#3\cpkk$'s with free, orientation-preserving involutions.  Their
quotients are exotic four-manifolds with $\pi_1=\Z/2\Z$ and $b_2=1$;
after reversing orientations, these give infinite collections of fake
projective planes.  The details of these two constructions are given in
\cite{definite2}.  Further related constructions, starting from
$\Sigma_g\times T^2$ and using equivariant torus surgeries, appear in
\cite{finitecyclic}.

The second comment concerns the real-symplectic nature of the symplectic
example.  The symplectic form in the normal-sum construction can be chosen
so that the free involution is anti-symplectic: on the two summands one
takes symplectic forms with opposite signs, and the involution exchanges
the two sides.  Thus the symplectic member of our construction is a real
symplectic four-manifold, in the sense of a symplectic four-manifold
$(X,\omega)$ equipped with an involution $\tau$ satisfying
$\tau^*\omega=-\omega$.  Since $\tau$ is fixed-point free, its real locus is
empty.

This suggests a real-symplectic geography problem: which simply connected
smooth four-manifolds admit real symplectic structures with empty real
locus?  The examples here, together with the constructions in
\cite{definite,definite2,finitecyclic}, show that such structures occur in
many homeomorphism classes of rational surfaces.  It is natural to compare
this with the complex-algebraic situation.  Simply connected minimal
complex surfaces of general type with $p_g=0$ and $K^2=1,2,3,4$ are known
in the homeomorphism classes of
$\cpk\#(8,7,6,5)\cpkk$, respectively, by work of Barlow, Lee--Park, and
Park--Park--Shin \cite{Barlow,LeePark,ParkParkShin3,ParkParkShin4}.
A fixed-point-free real structure on such a surface would give an
anti-holomorphic involution with empty real locus.

There is a simple parity obstruction in some cases.  If
$X$ is homeomorphic to $\cpk\#n\cpkk$ and admits a fixed-point-free
anti-holomorphic involution $c$, then the quotient is a smooth manifold
and $\chi(X)$ must be even. Since $\chi(\cpk\#n\cpkk)=n+3$, this forces
$n$ to be odd. The same parity condition also follows from the signature,
or from the Lefschetz number of $c$. In the
remaining cases $n=7$ and $n=5$ (corresponding to $K^2=2$ and $K^2=4$),
we do not know whether any of the known simply connected complex surfaces
of general type admits a fixed-point-free real structure. For comparison,
the standard del Pezzo representatives in these two homeomorphism classes
do admit fixed-point-free anti-holomorphic involutions; see for example
\cite{LeBrun}.

\end{document}